\documentclass{amsart}
\usepackage{psfrag,graphicx,epsf}
\usepackage{amsmath,amsthm,epsf,amscd,amssymb,verbatim}
\usepackage{palatino}
\usepackage[mathbf,mathcal]{euler}
\usepackage{epsfig}

\newtheorem{theorem}{Theorem}
\newtheorem{lemma}[theorem]{Lemma}
\newtheorem{cor}[theorem]{Corollary}

\theoremstyle{definition}
\newtheorem{definition}[theorem]{Definition}
\theoremstyle{remark}

\numberwithin{equation}{section} 

\newcommand{\ncmnd}{\newcommand}
\ncmnd{\nthm}{\newtheorem}

\nthm{thm}[theorem]{Theorem}
\nthm{lem}[theorem]{Lemma}

\theoremstyle{definition}
\nthm{defn}[theorem]{Definition}

\theoremstyle{remark}
\nthm{rem}[theorem]{Remark}
\nthm{examp}[theorem]{Example}

\newcommand{\catzero}{{\mbox{\sc cat($0$)}}}
\newcommand{\catk}{{\mbox{\sc cat($K$)}}}
\newcommand{\catone}{{\mbox{\sc cat($1$)}}}

\begin{document}

\title[Ishikawa iteration process on \catk\ spaces]{Ishikawa iteration process in \catk\ spaces}
\author{C. Jun}
\address{Department of Mathematics,
        University of Pennsylvania, Philadelphia PA, 19104}
\email{cjun@math.upenn.edu}

\begin{abstract}
In this paper, we establish $\Delta$-convergence results for Ishikawa iterations in complete \catk\ spaces.
\end{abstract}



\maketitle

\section{Introduction}
Let $(X,d)$ be a complete metric space and let $T$ be a mapping from $X$ to $X$. Then $T$ is called \emph{nonexpansive} if for all $x,y \in X$,
$$d(T(x), T(y) ) \leq d(x,y).$$ A point $x \in X$ is called a \emph{fixed point} of $T$ if $T(x) = x$. $Fix(T)$ denotes the set of fixed points of $T$. Kirk proved the existence of fixed points for nonexpansive mappings on \catzero\ spaces in \cite{k} and \catk\ spaces in \cite{k2}.

A \catk\ space is a metric space in which no triangle is fatter than the triangle with the same edge lengths in a model space, which is the 2-dimensional, complete, simply-connected space of constant curvature $K$ (see Definition \ref{de3}). A \catk\ space is a generalization of a simply-connected Riemannian manifold with sectional curvature $\leq K$; we will introduce generalized definitions of convergence and notations including the sum $\oplus$, which interpolates between a pair of points along a geodesic, in Section 2.

In a \catk\ space $X$, for $t_n, s_n \in [0,1]$ and $x_0 \in X$, the Ishikawa iteration $\{ x_n \}$ is defined by
\begin{equation}\label{e3}
x_{n+1} = t_n T (y_n) \oplus (1 - t_n) x_n , \;\; n \geq 0,
\end{equation}
where
$y_n = s_n T(x_n) \oplus (1-s_n) x_n$.

For a nonexpansive mapping $T$,
Dhompongsa and Panyanak in \cite{dp} obtained a $\Delta$-convergence result for Ishikawa iterations in complete \catzero\ spaces under the conditions $$\sum^{\infty}_{n=0} t_n(1-t_n) = \infty, \sum^{\infty}_{n=0} (1-t_n) s_n < \infty \;\;\textit{and}\;\; \limsup_{n} s_n < 1.$$
In \cite{pl}, a similar result was proved by Panyanak and Laokul under the other conditions $$\sum^{\infty}_{n=0} t_n(1-t_n) = \infty \;\; \textit{and} \;\; \sum^{\infty}_{n=0} t_n(1-t_n) s_n < \infty.$$
In this paper, we will obtain $\Delta$-convergence results for Ishikawa iterations in complete \catk\ spaces (see Theorem \ref{t1} and Theorem \ref{t2}).

In \cite{h}, He, Fang, L$\acute{\mathrm{o}}$pez and Li studied the $\Delta$-convergence of Mann iterations in complete \catk\ spaces with the condition $$\sum^{\infty}_{n=0} t_n(1-t_n) = \infty.$$  Since the Mann iteration is given by \eqref{e3} when $s_n =0$ for all $n$, we provide an alternative proof of the $\Delta$-convergence theorem for Mann iterations in complete \catk\ spaces.
\section{Preliminaries}
Let $(X,d)$ be a metric space. The open ball centered at $p$ with radius $r$ is denoted by $B_r(p)$. The closed ball centered at $p$ with radius $r$ is denoted by $B_r[p]$.

A curve $\gamma: I \to X$ is called a \emph{geodesic} if for any two $t, t^\prime \in I$, $d(\gamma(t),\gamma(t^\prime))=|t-t^\prime|$.
We denote by $[xy]$, a unit-speed geodesic $\gamma: I \to X$ from $x$ to $y$ defined on $I=[0,t]$, where $\gamma(0)=x$, $\gamma(t)=y$ and $t=d(x,y)$.
By $\triangle xyz$, we denote the geodesic triangle of geodesics $[x y]$,$[x z]$ and $[y z]$.

Let $C$ be a positive constant. A metric space $X$ is a \emph{geodesic space} if any two points are joined by a geodesic; and a \emph{C-geodesic space} if any two points with distance $< C$ are joined by a geodesic.
A set $Y \subset X$ is \emph{convex} if any two points $x,y \in Y$ can be joined by a geodesic and all geodesics joining them are contained in $Y$. If this condition holds for any two points in $Y$ with distance $<C$, $Y$ is said to be \emph{C-convex}.

For a constant $K$, we use $M_K$ to denote the 2-dimensional, complete, simply-connected space of constant curvature $K$.
Then $M_0 = \mathbb{E}^2$, $M_1 = \mathbb{S}^2$ and $M_{-1} = \mathbb{H}^2$.
Let $d_K$ be the metric of $M_K$.
$D_K$ denotes the diameter of $M_K$. Thus, $D_K = {\pi \over \sqrt{K} }$ if $K > 0$ and $D_K = \infty$ if $K \leq 0$.

A triangle $\triangle \widetilde{x}_1\widetilde{x}_2\widetilde{x}_3$ in $M_K$ is called a \emph{comparison triangle} for $\triangle x_1x_2x_3$ in $X$ if $d_K(\widetilde{x}_i,\widetilde{x}_j)=d(x_i,x_j)$ for $i,j \in \{1,2,3\}$.


\begin{definition}\label{de3}
Let $(X,d)$ be a metric space and let $K$ be a real constant. A $D_K$-geodesic space $X$ is a \emph{\catk\ space} if for any geodesic triangle $\triangle xy_1y_2$ of perimeter $< 2 D_K$, and its comparison triangle $\triangle \widetilde{x}\widetilde{y}_1\widetilde{y}_2$ in $M_K$, we have $$d(z_1,z_2) \leq d_K(\widetilde{z}_1,\widetilde{z}_2),$$ where $z_i$ is any point on $[x y_i]$ and $\widetilde{z}_i$ is the point on $[\widetilde{x} \widetilde{y}_i]$ such that $d_K(\widetilde{x},\widetilde{z}_i)=d(x,z_i)$ for $i \in \{1,2\}$.
\end{definition}

We now record a few lemmas about \catk\ spaces that we will need in the sequel.
\begin{lemma}\cite[Page 160]{bh}\label{l456}
Let $X$ be a \catk\ space. \\
$(1)$ For any two points $x,y$ in $X$ with distance less than $D_K$, there is a unique geodesic $[x y]$ connecting them. \\
$(2)$ Any ball in $X$ with radius less than $D_K/2$ is convex.
\end{lemma}

\begin{lemma}\cite[Page 178]{bh}\label{l123}
Let $(X,d)$ be a \catk\ space and let $F$ be a closed and $D_K$-convex subset of $X$. Then for each point $x \in X$ such that $d(x,F) < D_K/2$, there is a unique point $y \in F$ such that $d(x,y) = d(x,F)$.
\end{lemma}

Let $(X,d)$ be a \catk\ space and let $x, y \in X$ such that $d(x,y) < D_K$. Then
 $ t x \oplus (1-t) y $ denotes the unique point on $[x y]$ for $t \in [0,1]$ such that $d( x , tx \oplus (1-t) y ) = (1-t) d(x,y) $ and $ d( y , tx \oplus (1-t) y ) = t d(x,y) $.

\begin{lemma}\cite[Lemma 3.3]{p}\label{l}
For a positive number $C \leq \pi/2$, let $(X,d)$ be a \catone\ space and let $p, x, y \in X$ such that $d(p,x) \leq C$, $d(p,y) \leq C$ and $d(x,y) \leq C$.
Then for any $t \in [0,1]$, $$d( (1-t) p \oplus t x , (1-t) p \oplus t y ) \leq  {\sin t C  \over \sin C } d(x,y).$$
\end{lemma}

We can get the next lemma by following the proof of Prop. 3.1 in \cite{o} with $\varepsilon=\pi/4$.
\begin{lemma}\label{l0}
Let $(X,d)$ be a \catone\ space. Then there is a constant $k>0$ such that
$$d^2(x, ty \oplus (1-t)z ) \leq t d^2 ( x,y) + (1-t) d^2 ( x,z) - { k \over 2} t(1-t) d^2 (y,z)$$ for any $t \in [0,1]$
and any points $x, y, z \in X$ such that $d(x,y) \leq \pi/4$, $d(x,z) \leq \pi/4$ and $d(y,z) \leq  \pi/2$.
\end{lemma}

For a bounded sequence $\{x_n\}$ in $X$, define $$r(x,\{x_n\}) = \limsup_{ n \to \infty} d(x,x_n).$$ The \textit{asymptotic radius} of $\{x_n\}$ is defined by $$r(\{x_n\}) = \inf \{ r(x,\{x_n\}) : x \in X \}.$$
The \textit{asymptotic center} of $\{x_n\}$ is defined by $$A( \{x_n\} ) = \{ x \in X : r(x, \{x_n\} ) =  r(  \{x_n\} )    \}.$$
Now we can give a definition of $\Delta$-convergence, and list a few properties.
\begin{definition}
For a bounded sequence $\{x_n\}$ in $X$, the sequence $\{x_n\}$ is said to \emph{$\Delta$-converge to} $x \in X$ if $x$ is the unique asymptotic center of $\{u_n\}$ for every subsequence $\{u_n\}$ of $\{x_n\}$. In this case, we will write $\Delta-\lim_{n} x_n = x$ and call $x$ the \emph{$\Delta$-limit of} $\{x_n\}$.
\end{definition}
Then we have a lemma to show a property of a sequence which $\Delta$-converges.
\begin{lemma}\cite[Prop. 2.3]{h}\label{l000}
Let $(X,d)$ be a complete \catk\ space and let $p \in X$. Suppose that a sequence $\{ x_n \}$ $\Delta$-converges to $x$ such that $r(p,\{x_n\}) < D_K/2$. Then $$d(x,p) \leq \liminf_{ n \to \infty} d(x_n,p).$$
\end{lemma}

\begin{definition}
Let $(X,d)$ be a complete metric space and let $F$ be a nonempty subset of $X$. Then a sequence $\{x_n \}$ in $X$ is \emph{Fej$\acute{\textit{e}}$r monotone} with respect to $F$ if $$d( x_{n+1} , q) \leq d(x_n , q)$$ for all $n \geq 0$ and all $q \in F$.
\end{definition}

\begin{definition}\cite{h}\label{de5}
 For a sequence $\{x_n\}$ in $X$, a point $x \in X$ is a \emph{$\Delta$-cluster point} of $\{x_n\}$ if there exists a subsequence of $\{ x_n \}$ that $\Delta$-converges to $x$.
\end{definition}

With Definition \ref{de5}, we will know when a sequence $\{ x_n \}$ in $X$ $\Delta$-converges to a point of $F$ if $\{ x_n \}$ is  Fej$\acute{\textit{e}}$r monotone with respect to $F$.
\begin{lemma}\cite[Lemma 3.2]{h}\label{l00}
Let $X$ be a complete \catk\ space and let $F$ be a nonempty subset of $X$. Suppose that the sequence $\{ x_n \}$ of $X$ is  Fej$\acute{\textit{e}}$r monotone with respect to $F$ and the asymptotic radius $r(\{x_n\})$ of $\{x_n\}$ is less than $D_K/2$. If any $\Delta$-cluster point $x$ of $\{ x_n \}$ belongs to $F$, then the sequence $\{x_n\}$ $\Delta$-converges to a point of $F$.
\end{lemma}

\begin{lemma}\label{l1}\cite{z1}\cite{z2}
Suppose that $\{ a_n \}$ and $\{ b_n \}$ satisfy that
$$ a_n \geq 0,\; b_n \geq 0 \;\; \textit{and} \;\; a_{n+1} \leq (1 + b_n)a_n$$
for all $n \geq 0$.
If $\sum^{\infty}_{n=0} b_n$ converges, then $\lim_{ n \to \infty} a_n$ exists. Additionally, if there is a subsequence of $\{ a_n \}$ which converges to $0$, then $\lim_{ n \to \infty} a_n=0$.
\end{lemma}

\section{Ishikawa iteration process on \catk\ spaces}

\begin{lemma}\label{l3}\cite[Theorem 3.4]{p}
Let $X$ be a complete \catone\ space and let $T: X \to X$ be a nonexpansive mapping such that $F:=\mathrm{Fix}(T) \neq \emptyset$. Then $F$ is closed and $\pi$-convex.
\end{lemma}

\begin{lemma}\label{l5}
Let $X$ be a complete \catone\ space and let $T: X \to X$ be a nonexpansive mapping such that $F:=\mathrm{Fix}(T) \neq \emptyset$. If $\{x_n \}$ is defined by \eqref{e3} for $x_0 \in X$ such that $d(x_0,F) \leq \pi/4$, then there is a unique point $p$ in $F$ such that $x_n$, $y_n$, $T(x_n)$ and $T(y_n)$ are at distance $\leq d(x_0,p)$ from $p$.
\end{lemma}

\begin{proof}
Since $d(x_0,F) \leq \pi/4$, by Lemma \ref{l123} and \ref{l3}, there is a unique point $p$ in $F$ such that $d(x_0,p) = d(x_0,F)$.
By induction, we want to show that $$d(p, y_n ) \leq d(p, x_n) \leq d(p, x_0).$$
For $n=0$, since $T$ is nonexpansive, we have $d(p, T(x_0)  ) \leq d( p, x_0  ) \leq \pi/4$. Since $B_{\pi/4}[p]$ is convex, we get $$d(p,y_0) = d(p, s_0 T(x_0) \oplus (1-s_0) x_0) \leq d(p,x_0).$$
Suppose that $d(p, y_n) \leq d(p, x_n) \leq d(p, x_0)$.
Since $d(p, T(y_n)  ) \leq d( p, y_n  ) \leq \pi/4$ and $B_{\pi/4}[p]$ is convex, $$d(p,x_{n+1}) = d(p, t_n T(y_n) \oplus (1- t_n ) x_n) \leq d(p,x_n).$$
Since $d(p, T( x_{ n+1} ) ) \leq d( p, x_{n+1}  ) \leq \pi/4$, $$d(p, y_{n+1}) = d(p, s_{n+1} T(x_{n+1}) \oplus (1-s_{n+1}) x_{n+1}) \leq d(p,x_{n+1}).$$
Therefore $$d(p, y_{n+1}) \leq d(p, x_{n+1}) \leq d(p, x_{n}).$$
\end{proof}

We will prove Lemma \ref{l2} and Lemma \ref{l4}, which are obtained by following the proofs in \cite{pl}.

If $d(x_0,F)=0$, then $x_0 \in F$ and hence by definition \eqref{e3}, $x_n=x_0$ for all $n$. So we just consider the case $d(x_0,F)>0$.
\begin{lemma}\label{l2}
Let $X$ be a complete \catone\ space and let $T: X \to X$ be a nonexpansive mapping such that $F:=\mathrm{Fix}(T) \neq \emptyset$. If $\{x_n \}$ is defined by \eqref{e3} for $x_0 \in X$ such that $d(x_0,F) \leq \pi/4$, then
$$d( T (x_{n+1}), x_{n+1} ) \leq [ 1 + 4 {  C \over \sin C } t_n (1-t_n) s_n ] d( T (x_n) , x_n )$$ for all $n \geq 0$ where $C:= 2 d(x_0,F)$.
\end{lemma}
\begin{proof}
Since $d (T (x_n) , t_n T(x_n) \oplus (1-t_n)x_n ) = ( 1 - t_n )d(T(x_n), x_n)$ and $T$ is nonexpansive, we get
\begin{equation*}
\begin{split}
d( T (x_{n+1}), x_{n+1} ) &\leq d(T (x_{n+1}) , T( t_n T(x_n) \oplus (1-t_n) x_n) ) \\ &+ d( T( t_n T(x_n) \oplus (1-t_n) x_n) , T (x_n) ) \\
&+ d (T (x_n) , t_n T(x_n) \oplus (1-t_n)x_n ) \\ &+ d( t_n T(x_n) \oplus (1-t_n) x_n , x_{n+1} ) \\
&\leq 2 d( t_n T(x_n) \oplus (1-t_n) x_n , x_{n+1} ) \\ &+ d(  t_n T(x_n) \oplus (1-t_n) x_n ,  x_n ) \\  &+ ( 1 - t_n )d(T(x_n), x_n). \\
\end{split}
\end{equation*}
Since $ d(  t_n T(x_n) \oplus (1-t_n) x_n ,  x_n ) =  t_n d(T(x_n), x_n)$,
it becomes
\begin{equation*}
d( T (x_{n+1}), x_{n+1} ) \leq 2 d( t_n T(x_n) \oplus (1-t_n) x_n , x_{n+1} ) + d(T(x_n), x_n).
\end{equation*}
Note that $d(x_n,T(x_n)) \leq C$, $d(x_n,T(y_n)) \leq C$ and $d(T(x_n),T(y_n)) \leq C$ from Lemma \ref{l5}. Since $d( x_n ,y_n) = s_n d( T (x_n), x_n)$ and $C \leq \pi /2$, by Lemma \ref{l},
\begin{equation*}
\begin{split}
d( T (x_{n+1}), x_{n+1} )
&\leq 2 {\sin t_n C  \over \sin C } d( T(x_n) , T(y_n) ) + d(T(x_n), x_n) \\
&\leq 2 { t_n C \over \sin C } d( x_n ,y_n) + d(T(x_n), x_n) \\
&= (1 + 2 { t_n C \over \sin C } s_n) d( T (x_n), x_n). \\
\end{split}
\end{equation*}
Then, multiplying by $(1 - t_n)$, we have
\begin{equation}\label{e1}
(1 - t_n) d(T(x_{n+1}), x_{n+1} ) \leq [ 1 - t_n  + 2 { C \over \sin C } t_n(1-t_n)s_n ] d( T (x_n), x_n).
\end{equation}
Also, since $d (T (y_n) , t_n T(y_n) \oplus (1-t_n)y_n ) = ( 1 - t_n )d(T(y_n), y_n)$ and $T$ is nonexpansive, we get
\begin{equation*}
\begin{split}
d( T (x_{n+1}), x_{n+1} ) &\leq d(T (x_{n+1}) , T( t_n T(y_n) \oplus (1-t_n) y_n) ) \\ &+ d( T( t_n T(y_n) \oplus (1-t_n) y_n) , T (y_n) ) \\
&+ d (T (y_n) , t_n T(y_n) \oplus (1-t_n)y_n ) \\ &+ d( t_n T(y_n) \oplus (1-t_n) y_n , x_{n+1} ) \\
&\leq 2 d( t_n T (y_n) \oplus (1-t_n) y_n , x_{n+1} ) \\ &+ d(  t_n T(y_n) \oplus (1-t_n) y_n ,  y_n ) \\  &+ ( 1 - t_n )d(T(y_n), y_n). \\
\end{split}
\end{equation*}
Since $d(  t_n T(y_n) \oplus (1-t_n) y_n ,  y_n ) =  t_n d(T(y_n), y_n)$,
it becomes
\begin{equation*}
d( T (x_{n+1}), x_{n+1} ) \leq 2 d( t_n T (y_n) \oplus (1-t_n) y_n , x_{n+1} ) + d(T(y_n), y_n).
\end{equation*}
Since $d(x_n,T(y_n)) \leq C$, $d(y_n,T(y_n)) \leq C$ and $d(x_n,y_n) \leq C$, by Lemma \ref{l},
\begin{equation*}
\begin{split}
d( T (x_{n+1}), x_{n+1} )
&\leq 2 {\sin (1-t_n) C  \over \sin C } d( x_n , y_n ) + d(T(y_n), y_n) \\
&\leq 2 { (1-t_n) C \over \sin C } d( x_n ,y_n) + d(T(y_n), y_n) \\
&\leq 2 { (1-t_n) C \over \sin C } d( x_n ,y_n) + d(T(y_n), T(x_n)) + d(T (x_n) , y_n) \\
&\leq 2 { (1-t_n) C \over \sin C } d( x_n ,y_n) + d( y_n,  x_n) + d(T (x_n) , y_n) \\
&\leq 2 { (1-t_n) C \over \sin C } s_n d( x_n ,T (x_n) ) + s_n d( T (x_n),  x_n) \\ &+ (1-s_n) d(T (x_n) , x_n) \\
&= [1 + 2 { (1-t_n) C \over \sin C } s_n ]d( T (x_n) , x_n ).  \\
\end{split}
\end{equation*}
Then we have
\begin{equation}\label{e2}
t_n d( T (x_{n+1}), x_{n+1} ) \leq [t_n + 2 {  C \over \sin C } t_n (1-t_n) s_n ]d( T (x_n) , x_n ).
\end{equation}
From \eqref{e1} and \eqref{e2}, we get
$$d( T (x_{n+1}), x_{n+1} ) \leq [ 1 + 4 {  C \over \sin C } t_n (1-t_n) s_n ] d( T (x_n) , x_n ).$$
\end{proof}

\begin{lemma}\label{l4}
Let $X$ be a complete \catone\ space and let $T: X \to X$ be a nonexpansive mapping such that $F:=\mathrm{Fix}(T) \neq \emptyset$. Suppose that $\{t_n \}$ and $\{s_n \}$ satisfy that
$$\sum^{\infty}_{n=0} t_n(1-t_n) = \infty  \;\; \textit{and} \;\;
\sum^{\infty}_{n=0} t_n(1-t_n) s_n < \infty.$$
If $\{x_n \}$ is defined by \eqref{e3} for $x_0 \in X$ such that $d(x_0,F) \leq \pi/4$, then $$\lim_{ n \to \infty} d(T (x_n) , x_n ) = 0.$$
\end{lemma}
\begin{proof}
By Lemma \ref{l1} and \ref{l2}, $\lim_{ n \to \infty} d(T (x_n) , x_n )$ exists.
Let $p$ be the unique point in $F$ such that $d(x_0,p) = d(x_0,F)$.
By Lemma \ref{l0}, we get
\begin{equation*}
\begin{split}
d^2(p,x_{n+1}) &= d^2(p, t_n T (y_n) \oplus (1 - t_n) x_n ) \\
&\leq t_n d^2(p, T (y_n) ) + (1-t_n) d^2 (p, x_n) - { k \over 2} t_n(1-t_n) d^2 (T (y_n) , x_n). \\
\end{split}
\end{equation*}
Then since $T$ is nonexpansive, we have
\begin{equation}\label{e4}
d^2(p,x_{n+1}) \leq t_n d^2(p,  y_n ) + (1-t_n) d^2 (p, x_n) - { k \over 2} t_n(1-t_n) d^2 (T (y_n) , x_n).
\end{equation}
By Lemma \ref{l0}, also we get
\begin{equation*}
\begin{split}
d^2(p,y_{n}) &= d^2(p, s_n T (x_n) \oplus (1 - s_n) x_n ) \\
&\leq s_n d^2(p, T (x_n) ) + (1-s_n) d^2 (p, x_n) - { k \over 2} s_n(1-s_n) d^2 (T (x_n) , x_n). \\
\end{split}
\end{equation*}
Since $T$ is nonexpansive,
\begin{equation}\label{e5}
\begin{split}
d^2(p,y_{n})
&\leq s_n d^2(p,  x_n ) + (1-s_n) d^2 (p, x_n) - { k \over 2} s_n(1-s_n) d^2 (T (x_n) , x_n)  \\
&\leq d^2(p,  x_n ).   \\
\end{split}
\end{equation}
By \eqref{e4} and \eqref{e5}, we obtain $$ d^2(p,x_{n+1}) \leq d^2(p,x_{n}) - { k \over 2} t_n(1-t_n) d^2 (T (y_n) , x_n).$$
This implies
\begin{equation}\label{e6}
{ k \over 2}  \sum^{\infty}_{n = 0} t_n(1-t_n) d^2 (T (y_n) , x_n) \leq d^2(p,x_0) < \infty.
\end{equation}
Since $\sum^{\infty}_{n=0} t_n(1-t_n) s_n < \infty$, from \eqref{e6}, we get
$$\sum^{\infty}_{n = 0} t_n(1-t_n) [ d^2 (T (y_n) , x_n) + s_n]  < \infty.$$
Since $\sum^{\infty}_{n=0} t_n(1-t_n) = \infty$, it implies
$$ \liminf_{ n \to \infty} [ d^2 (T (y_n) , x_n) + s_n] =0.$$
Then there exists a subsequence $\{n_k \}$ of $\{n \}$ such that
\begin{equation}\label{e7}
 \lim_{k \to  \infty} d (T (y_{n_k}) , x_{n_k}) = 0 \;\;\;\;\mathrm{and} \;\;  \lim_{k \to  \infty} s_{n_k} = 0.
 \end{equation}
Also,
\begin{equation*}
\begin{split}
 d (T (x_{n_k}) , x_{n_k}) &\leq d (T (x_{n_k}) , T (y_{n_k})) +  d (T (y_{n_k}) , x_{n_k})  \\
 & \leq  d ( x_{n_k} ,  y_{n_k})  + d (T (y_{n_k}) , x_{n_k})  \\
 & = s_{n_k} d (T (x_{n_k}) , x_{n_k}) + d (T (y_{n_k}) , x_{n_k}).
\end{split}
\end{equation*}
Then it becomes
\begin{equation}\label{e8}
(1- s_{n_k}) d (T (x_{n_k}) , x_{n_k}) \leq d (T (y_{n_k}) , x_{n_k}).
\end{equation}
From \eqref{e7} and \eqref{e8}, we get
\begin{equation}\label{e14}
\lim_{k \to \infty} d (T (x_{n_k}) , x_{n_k}) = 0.
\end{equation}
Since $\lim_{ n \to \infty} d(T (x_n) , x_n )$ exists, \eqref{e14} implies that $$\lim_{ n \to \infty} d(T (x_n) , x_n )=0.$$
\end{proof}

We will prove Theorem \ref{t1}, which is obtained by following a first part of the proof in \cite[Theorem 3.1]{h}.

\begin{theorem}\label{t1}
Let $X$ be a complete \catk\ space and let $T: X \to X$ be a nonexpansive mapping such that $F:=\mathrm{Fix}(T) \neq \emptyset$.
Suppose that $\{x_n\}$ is defined by \eqref{e3} under the conditions $$\sum^{\infty}_{n=0} t_n(1-t_n) = \infty
\;\; \textit{and} \;\;
\sum^{\infty}_{n=0} t_n(1-t_n) s_n < \infty.$$ Then, for each $x_0 \in X$ with $d(x_0, F ) < D_K /4 $, the sequence $\{x_n\}$ $\Delta$-converges to a point of $F$.
\end{theorem}
\begin{proof}
Rescaling the metric by $1/\sqrt{K}$, we may assume that $K=1$.

Set $F_0:= F \cap B_{\pi/2}(x_0)$.
For any $q \in F_0$, since $d(T(x_0) ,q ) \leq d(x_0,q ) $ and since the open ball $B_r(q)$ in $X$ with radius $r < \pi/2$ is convex, we have $$d( y_0 , q) = d( s_0 T(x_0) \oplus (1-s_0) x_0 , q ) \leq  d(x_0,q ).$$ Similarly, since $d(T(y_0) ,q ) \leq d(y_0,q )$ and since the open ball $B_r(q)$ is convex,
we have $$d(x_1 , q) = d( t_0 T (y_0) \oplus (1 - t_0) x_0 ,q) \leq d(x_0 ,q).$$
Using mathematical induction, we can easily get that $$d( x_{n+1} , q ) \leq d( x_{n} , q ) \leq  d( x_0 , q )$$ for all $n \geq 0$.
Therefore the sequence $\{x_n\}$ is Fej$\acute{\mathrm{e}}$r monotone with respect to $F_0$.

Let $p$ be the unique point in $F$ such that $d(x_0,p) = d(x_0,F)$. Then $p \in F_0$. Also
we get
\begin{equation}\label{e11}
d( x_{n+1} , p ) \leq d( x_{n} , p ) \leq  d( x_0 , p ) < \pi/4
\end{equation}
for all $n \geq 0$. This means that the asymptotic radius $r(\{x_n \})$ of $\{x_n\}$ is less than $\pi/4$.

By Lemma \ref{l00},
we only need to show that for each point $x$ such that there exists a subsequence $\{ x_{n_k} \}$ of $\{ x_n \}$ which $\Delta$-converges to $x$, it belongs to $F_0$. From \eqref{e11}, note that $r(p, \{x_n\}) \leq d(x_0,p) < \pi/4$.

By Lemma \ref{l000}, we obtain
$$d(x,x_0) \leq d(x,p) + d(x_0,p) \leq \liminf_{ k } d(x_{n_k},p) + d(x_0,p)  < \pi/2,$$
that is,
\begin{equation}\label{e12}
x \in B_{\pi/2}(x_0).
\end{equation}
By Lemma \ref{l4}, we have
\begin{equation*}
\begin{split}
 \limsup_{k} d( T(x) , x_{n_k}) &\leq \limsup_{k} d( T(x) , T (x_{n_k})) + \limsup_{k} d( T (x_{n_k}) , x_{n_k})  \\
 &\leq \limsup_{k} d( x ,  x_{n_k}).
\end{split}
\end{equation*}
This implies that $T(x) \in A(\{x_{n_k}\})$ and $T(x)=x$. Therefore $x \in F$. With \eqref{e12}, $x$ belongs to $F_0$. By Lemma \ref{l00}, it is proved.
\end{proof}
Since the following corollary is a direct result by letting $K=0$ in Theorem \ref{t1}, Theorem \ref{t1} is the extended result of the corresponding one in \cite{pl}.
\begin{cor}
Let $X$ be a complete \catzero\ space and let $T: X \to X$ be a nonexpansive mapping such that $F:=\mathrm{Fix}(T) \neq \emptyset$.
Suppose that $\{x_n\}$ is defined by \eqref{e3} under the conditions $$\sum^{\infty}_{n=0} t_n(1-t_n) = \infty
\;\; \textit{and} \;\;
\sum^{\infty}_{n=0} t_n(1-t_n) s_n < \infty.$$ Then, for each $x_0 \in X$, the sequence $\{x_n\}$ $\Delta$-converges to a point of $F$.
\end{cor}

For the conditions $\sum^{\infty}_{n=0} t_n(1-t_n) = \infty$, $\sum^{\infty}_{n=0} (1-t_n) s_n < \infty$ and $\limsup_{n} s_n < 1$, we will get the following lemma, which is an analog of Lemma 2.12 in \cite{dp}.
\begin{lemma}\label{l7}
Let $X$ be a complete \catone\ space and let $T: X \to X$ be a nonexpansive mapping such that $F:=\mathrm{Fix}(T) \neq \emptyset$. Suppose that $\{t_n \}$ and $\{s_n \}$ satisfy that
$$\sum^{\infty}_{n=0} t_n(1-t_n) = \infty, \sum^{\infty}_{n=0} (1-t_n) s_n < \infty \;\; \textit{and} \;\; \limsup_{n} s_n < 1.$$
If $\{x_n \}$ is defined by \eqref{e3} for $x_0 \in X$ such that $d(x_0,F) < \pi/4$, then $$\lim_{ n \to \infty} d(T (x_n) , x_n ) = 0.$$
\end{lemma}
\begin{proof}
From Equation \eqref{e6} in Lemma \ref{l4}, $${ k \over 2}  \sum^{\infty}_{n = 0} t_n(1-t_n) d^2 (T (y_n) , x_n) \leq d^2(p,x_0) < \infty$$ for the unique point $p$ in $F$ such that $d(x_0,p) = d(x_0,F)$. Since $\sum^{\infty}_{n=0} t_n(1-t_n)$ diverges, it means that $$\liminf_n d^2 (T (y_n) , x_n) = 0$$ and then
\begin{equation}\label{e9}
\liminf_n d (T (y_n) , x_n) = 0.
\end{equation}
Since $T$ is nonexpansive and $d( x_n, y_n )=s_n d( T(x_n) , x_n )$,
\begin{equation*}
\begin{split}
d( T(x_n) ,x_n) &\leq d( T(x_n),T(y_n) ) + d( T(y_n) , x_n) \\
& \leq d( x_n, y_n ) + d( T(y_n) , x_n) \\
& = s_n d( T(x_n) , x_n ) + d( T(y_n) , x_n).
\end{split}
\end{equation*}
Then we get $$d( T(x_n) ,x_n ) \leq { 1 \over 1 - s_n} d ( T(y_n) , x_n).$$
By \eqref{e9},
\begin{equation}\label{e10}
\liminf_n d (T (x_n) , x_n) = 0.
\end{equation}

Since $d(x_0,F) < \pi/4$, by Lemma \ref{l5}, $x_n$ and $T(y_n)$ are in the open ball centered at $T(x_{n+1})$ with radius $< \pi/2$.
Since this open ball is convex, we have
\begin{equation*}
\begin{split}
d( T(x_{n+1}) , x_{n+1} ) &\leq t_n d( T(x_{n+1}) , T(y_n) ) + (1-t_n) d( T(x_{n+1}) , x_n ) \\
&\leq t_n d( x_{n+1} , y_n ) + (1-t_n) [d( T(x_{n+1}) , x_{n+1} ) + d( x_{n+1} , x_n ) ] \\
&\leq t_n d( x_{n+1} , y_n ) + (1-t_n) [d( T(x_{n+1}) , x_{n+1} ) + t_n d( T(y_n) , x_n ) ]. \\
\end{split}
\end{equation*}
Dividing by $t_n$, this becomes
$$ d( T(x_{n+1}) , x_{n+1} ) \leq d( x_{n+1} , y_n ) + (1-t_n) d( T(y_n) , x_n ).$$
By Lemma \ref{l5}, $T(y_n)$ and $x_n$ are in the open ball centered at $y_n$ with radius $< \pi/2$, which is convex.
This yields
\begin{equation*}
\begin{split}
  d( T(x_{n+1}) , x_{n+1} ) &\leq t_n d( T(y_n) , y_n ) + (1-t_n) d(x_n,y_n) + (1-t_n) d( T(y_n) , x_n ).  \\
\end{split}
\end{equation*}
Since $T(x_n)$ and $x_n$ are in the open ball centered at $T(y_n)$ with radius $< \pi/2$, we have
\begin{equation*}
\begin{split}
 d( T(x_{n+1}) , x_{n+1} ) &\leq t_n [ s_n d( T(y_n) , T(x_n)  ) + ( 1- s_n) d( T(y_n) , x_n)  ] \\ &+ (1-t_n) d(x_n,y_n) + (1-t_n) d( T(y_n) , x_n ).  \\
\end{split}
\end{equation*}
Since $T$ is nonexpansive and $d( x_n, y_n )=s_n d( T(x_n) , x_n )$, we get
\begin{equation*}
\begin{split}
 d( T(x_{n+1}) , x_{n+1} ) &\leq (1-t_n + t_n s_n ) d(x_n,y_n) + (1 - t_n s_n) d( T(y_n) , x_n ) \\
 &\leq s_n (1-t_n + t_n s_n ) d(x_n, T(x_n) ) \\
 &+ (1 - t_n s_n) [ d( T(y_n) , T(x_n)  ) + d( T(x_n) , x_n ) ]\\
 &\leq [s_n (1-t_n + t_n s_n )  + (1 - t_n s_n)(1 + s_n)    ] d(x_n, T(x_n) ) \\
 &= [ 1 + 2 s_n ( 1 - t_n) ] d(x_n, T(x_n) ).
\end{split}
\end{equation*}
Then we get the following inequality
\begin{equation}\label{e13}
d( T(x_{n+1}) , x_{n+1} ) \leq [ 1 + 2 s_n ( 1 - t_n) ] d(x_n, T(x_n) ).
\end{equation}
Since $\sum s_n ( 1 - t_n)$ converges, applying Lemma \ref{l1} to \eqref{e13}, $\lim_{ n \to \infty} d(T (x_n) , x_n ) $ exists. By \eqref{e10}, it is equal to zero.
\end{proof}
By following the same proof of Theorem \ref{t1} and using Lemma \ref{l7}, we obtain
\begin{theorem}\label{t2}
Let $X$ be a complete \catk\ space and let $T: X \to X$ be a nonexpansive mapping such that $F:=\mathrm{Fix}(T) \neq \emptyset$. Suppose that $\{x_n\}$ is defined by \eqref{e3} under the conditions
$$\sum^{\infty}_{n=0} t_n(1-t_n) = \infty, \sum^{\infty}_{n=0} (1-t_n) s_n < \infty \;\; \textit{and} \;\; \limsup_{n} s_n < 1.$$
Then, for each $x_0 \in X$ with $d(x_0, F ) < D_K /4 $, the sequence $\{x_n\}$ $\Delta$-converges to a point of $F$.
\end{theorem}

\section*{Acknowledgements}
\small The author would like to express his gratitude to Prof. R. Ghrist for his support.
He gratefully acknowledge support from the ONR Antidote MURI project, grant no. N00014-09-1-1031.

\bibliography{thesisbib}

\end{document}